\documentclass[12pt,draft]{article} 
\usepackage{amsmath,amssymb,amsthm,amsfonts,bm}
\usepackage{enumerate,color}

\makeatletter
\def\@cite#1#2{[{{\bfseries #1}\if@tempswa , #2\fi}]}
\renewcommand{\section}{%
\@startsection{section}{1}{\z@}
{0.5truecm plus -1ex minus -.2ex}%
{1.0ex plus .2ex}{\bfseries\large}}
\def\@seccntformat#1{\csname the#1\endcsname.\ }
\makeatother

\numberwithin{equation}{section} 
\newtheorem{thm}{Theorem}[section]
\newtheorem{corollary}[thm]{Corollary}
\newtheorem{lem}[thm]{Lemma}

\theoremstyle{definition}

\newtheorem{remark}{Remark}[section]

\newcommand{\ep}{\varepsilon}
\newcommand{\pa}{\partial}
\newcommand{\Rn}{\mathbb{R}^n}

\newcommand{\ol}[1]{\overline{#1}}

\newcommand{\tmax}{T_{\rm max,\lambda}}
\newcommand{\lp}[2]{\left\Vert{#2}\right\Vert_{L^{#1}(\Omega)}}
\newcommand{\wmp}[2]{\Vert{#2}\Vert_{W^{#1}(\Omega)}}
\newcommand{\cd}{(\cdot,t)}
\newcommand{\into}{\int_\Omega}
\newcommand{\obar}{\overline{\Omega}}
\newcommand{\utau}{u_\lambda}
\newcommand{\vtau}{v_\lambda}
\newcommand{\lambdan}{\lambda_n}
\newcommand{\lambdanj}{\lambda_{n_j}}
\newcommand{\utaunj}{u_{\lambda_{n_j}}}
\newcommand{\vtaunj}{v_{\lambda_{n_j}}}

\newcommand{\lambdaz}{\lambda_0}
\newcommand{\itau}{\lambda \in (0,\lambdaz)}
\newcommand{\uinit}{u_{\rm init}}
\newcommand{\vinit}{v_{\rm init}}
\newcommand{\uobar}{\overline{u}}
\newcommand{\vobar}{\overline{v}}

\newcommand{\tlam}{T_\lambda}
\newcommand{\wlam}{W_{\lambda}}

\begin{document}
\footnote[0]
    {2010{\it Mathematics Subject Classification}\/. 
    Primary: 35A09; Secondary:  35K51, 35J15, 92C17.
    }
\begin{center}
    \large{{\bf The fast signal diffusion limit in a Keller--Segel system}} 
\end{center}
\vspace{5pt}
\begin{center}
    Masaaki Mizukami\footnote{Partially supported by 
    JSPS Research Fellowships for Young Scientists (No.\ 17J00101).} 
   \footnote[0]{
    E-mail: 
    {\tt masaaki.mizukami.math@gmail.com} 
    }\\
    \vspace{12pt}
    \small{Department of Mathematics, 
    Tokyo University of Science\\
    1-3, Kagurazaka, Shinjuku-ku, 
    Tokyo 162-8601, Japan}\\
    \vspace{2pt}
\end{center}
\begin{center}    
    \small \today
\end{center}
\vspace{2pt}
\newenvironment{summary}
{\vspace{.5\baselineskip}\begin{list}{}{%
     \setlength{\baselineskip}{0.85\baselineskip}
     \setlength{\topsep}{0pt}
     \setlength{\leftmargin}{12mm}
     \setlength{\rightmargin}{12mm}
     \setlength{\listparindent}{0mm}
     \setlength{\itemindent}{\listparindent}
     \setlength{\parsep}{0pt}
     \item\relax}}{\end{list}\vspace{.5\baselineskip}}
\begin{summary}
{\footnotesize {\bf Abstract.}
This paper deals with 
convergence of a solution for 
the parabolic-parabolic Keller--Segel system 
\begin{align*}
\begin{cases}
  (u_\lambda)_t = \Delta u_\lambda - \chi \nabla \cdot (u_\lambda \nabla v_\lambda) & 
  \mbox{in} \ \Omega \times (0,\infty),
\\
  \lambda (v_\lambda)_t = \Delta v_\lambda - v_\lambda +u_\lambda  
 &\mbox{in} \ \Omega\times (0,\infty)
\end{cases}
\end{align*}
to that for the parabolic-elliptic Keller--Segel system 
\begin{align*}
\begin{cases}
 u_t = \Delta u - \chi \nabla \cdot (u \nabla v)&
 \mbox{in} \ \Omega \times (0,\infty),
\\
 0= \Delta v -v +u 
&
 \mbox{in} \ \Omega\times (0,\infty) 
\end{cases}
\end{align*}
as $\lambda \searrow 0$, 
where $\Omega$ is a bounded domain in $\mathbb{R}^n$ ($n\ge 2$) with smooth boundary, 
$\chi, \lambda>0$ are constants. 
In chemotaxis systems 
parabolic-elliptic systems often provided some guide to 
methods and results for 
parabolic-parabolic systems. 
However, 
there have not been rich results on the relation 
between parabolic-elliptic systems and parabolic-parabolic systems. 
Namely, it still remains to analyze on the following question except some cases: 
{\it Does a solution of the parabolic-parabolic system converge to 
that of the parabolic-elliptic system as $\lambda \searrow 0$?} 
In the case that $\Omega$ is the whole space $\mathbb{R}^n$, or 
$\Omega$ is a bounded domain and $\chi$ is a strong signal sensitivity, 
some positive answer was shown 
in the author's previous paper (Math.\ Nachr., to appear). 
Therefore one can expect a positive answer to this question also in the Keller--Segel system in a bounded domain $\Omega$ in some cases.
This paper gives some positive answer in 
the 2-dimensional and the higher-dimensional Keller--Segel system. 
}
\end{summary}
\vspace{10pt} 

\newpage
%
%

\section{Introduction}

The subject of this work is to construct a new approach to 
a parabolic-elliptic Keller--Segel system from its parabolic-parabolic case, 
and to use the parabolic-parabolic case as a step to establish new results 
in the parabolic-elliptic case.  
In this paper our aim is, by considering that the parabolic-elliptic system is 
as a limit of its parabolic-parabolic case, to establish a result such that  
only dealing with the parabolic-parabolic Keller--Segel system is enough to 
obtain new properties for solutions of its parabolic-elliptic case. 
As a related work, in the study of a chemotaxis system with 
signal-dependent sensitivity, some result on this subject 
has already been obtained (\cite{Mizukami_fast_sensi}); 
however, in this study we could not 
attain a result on a  
minimal Keller--Segel system 
from a technical reason. 
Thus the subject of this paper is a challenging problem for 
a progress of the chemotaxis system.

\smallskip
Before an introduction of a problem in this paper, we will recall some related works 
on the chemotaxis system. Here chemotaxis is the property such that 
species move towards higher concentration of a chemical substance 
when they plunge into hunger. 
Keller--Segel \cite{K-S,K-S_2} studied the migration of the species which have chemotaxis, and proposed the following problem: 
\[ 
  u_t = \Delta u - \nabla \cdot (u\chi(v)\nabla v), 
\quad 
  \lambda v_t =  \Delta v - v +u 
\quad \mbox{in} \ \Omega \times (0,\infty), 
\] 
where $\Omega\subset \mathbb{R}^n$ ($n\in \mathbb{N}$) is a bounded domain, 
$\lambda =0$ (the parabolic-elliptic system) or $\lambda>0$ (the parabolic-parabolic system) is a constant and  
$\chi$ is a function. 
This problem is called a {\it chemotaxis system}, and especially, 
is called a (minimal) {\it Keller--Segel system} in the case that $\chi$ is a constant function. 
About the Keller--Segel system, Nanjundiah \cite{Nanjundiah_1973} first asserted that 
we could expect existence of a blow-up solution to the Keller--Segel system. 
Moreover, Childress--Percus \cite{Childress-Percus_81} claimed 
the following conjecture: 
\begin{itemize}
\setlength{\itemsep}{0cm}
\item In the $1$-dimensional setting, global existence holds. 
\item In the $2$-dimensional setting, there is a critical number $c$ 
such that  if an initial data $\uinit$ satisfies $\lp{1}{\uinit} < c$ 
then global existence holds, and for any $m>c$ there are initial data $\uinit,\vinit$
such that $\lp{1}{\uinit} = m$ and the corresponding solution blows up in finite time. 
\item In the higher-dimensional setting, there are many blow-up solutions. 
\end{itemize} 
Here we first focus on the 2-dimensional setting. 
The study of the 2-dimensional Keller--Segel system is supported by 
the interaction between the parabolic-elliptic case and 
the parabolic-parabolic case. 
In order to verify the Childress--Percus conjecture Nagai \cite{Nagai_1995} tried to deal with 
the parabolic-elliptic case which is a simplified problem of 
the parabolic-parabolic Keller--Segel system, 
and shown that, in the radial setting, $8\pi$ is the critical value 
in the Childress--Percus conjecture. 
Subsequently, Nagai--Senba--Yoshida \cite{Nagai-Senba-Yoshida} established  
global existence and boundedness of {\it radial} solutions 
in the parabolic-parabolic Keller--Segel system 
under the condition that $\lp{1}{\uinit} < 8\pi$, 
and also obtained existence of global bounded {\it nonradial} solutions to 
the parabolic-parabolic system 
under the condition that $\lp{1}{\uinit}< 4\pi$. 
Here Senba--Suzuki \cite{Senba-Suzuki_Kokyuroku} asserted that arguments in 
proofs of these results could also be applied to the parabolic-elliptic case, 
which meant that global existence and boundedness of solutions to 
the parabolic-elliptic system were shown 
under the condition that $\lp{1}{\uinit} < 4\pi$. 
Therefore in the both cases of the parabolic-elliptic system and 
the parabolic-parabolic system, 
$8\pi$ is the critical value in the Childress--Percus conjecture in the radial case, 
and $4\pi$ is the critical value in the nonradial case. 
Indeed, existence of blow-up solutions such that 
$\lp{1}{\uinit}$ is larger than the critical value was shown 
(\cite{Herrero-Velazques_1997,Horstmann-Wang,Mizoguchi-Winkler,Nagai_1995,Nagai_2001}): 
The radial parabolic-elliptic case was treated by a combination of the results in \cite{Herrero-Velazques_1997,Nagai_1995};  
the nonradial parabolic-elliptic case is in \cite{Nagai_2001}; 
the radial parabolic-parabolic case can be found in \cite{Mizoguchi-Winkler};  
the nonradial parabolic-parabolic case is in \cite{Horstmann-Wang}. 
Moreover, related works which deal with 
blow-up asymptotics of solutions to the parabolic-elliptic case 
can be found in \cite{Herrero-Velazques_1997,Senba_2007,Senba-Suzuki_2001} and 
to the parabolic-parabolic case are in \cite{Mizoguchi-Souplet_2014,Nagai-Senba-Suzuki_2000}. 
In summary, 
in the 2-dimensional setting, 
the study of the Keller--Segel system 
was developed by the interaction between the parabolic-elliptic system and the
parabolic-parabolic system,  
and it is shown that the Childress--Percus conjecture is true. 
On the other hand, the other dimensional cases have also been studied 
only in the parabolic-parabolic system, 
and it is shown that the Childress--Percus conjecture is valid also 
in the other dimensional cases; 
in the 1-dimensional setting Osaki--Yagi \cite{Osaki-Yagi} showed 
global existence and boundedness of classical solutions; 
in the higher-dimensional case Winkler \cite{Winkler_2013_blowup} obtained that 
for all $m>0$ there are initial data $\uinit,\vinit$ such that $\lp{1}{\uinit} = m$ and 
the corresponding solution blows up in finite time. 
Here global existence of bounded solutions to the higher-dimensional 
parabolic-parabolic Keller--Segel system also holds under some smallness condition 
for initial data $\uinit,\vinit$ with respect to some Lebesgue norm; 
Winkler first established 
global existence and boundedness in the higher-dimensional 
parabolic-parabolic  Keller--Segel system under the condition that $\lp{p}{\uinit}$ and 
$\lp{q}{\nabla \vinit}$ are sufficiently small with some $p>\frac n2$ and $q>n$; 
Cao \cite{Xinru_higher} obtained 
global existence of bounded solutions to the parabolic-parabolic system 
under the smallness conditions for initial data in optimal spaces:  
$\lp{\frac{n}{2}}{\uinit}$ and 
$\lp{n}{\nabla \vinit}$ are small enough. 

\smallskip
As we mentioned before, 
the interaction between the parabolic-elliptic system and 
the parabolic-parabolic system 
made progress on researches of the Keller--Segel system. 
The similar things occurred in the study of the chemotaxis system with signal-dependent sensitivity which is the case that $\chi$ is a function. 
In the parabolic-elliptic system with $\chi(v)=\frac{\chi_0}{v}$ ($\chi_0>0$) 
Nagai--Senba \cite{Nagai-Senba_1998} first showed that 
if $n=2$, or $n\ge 3$ and $\chi_0<\frac{2}{n-2}$ then a radial solution is global and bounded, 
and if $n\ge 3$ and $\chi_0>\frac{2n}{n-2}$ then there exists some initial data such that 
a radial solution blows up in finite time. 
In the nonradial case Biler \cite{Biler_1999} obtained global existence of solutions to 
the parabolic-elliptic system with $\chi(v)=\frac{\chi_0}{v}$ ($\chi_0>0$) 
under the conditions that $n=2$ and $\chi_0\le 1$, or $n\ge 3$ and $\chi_0<\frac{2}{n}$.  
Thanks to these results, we can expect that conditions for global existence in 
the above system were 
determined by a dimension of a domain and  
a smallness of $\chi$ in some sense. 
Indeed, global existence and boundedness of 
solutions to the parabolic-elliptic system 
with $\chi(v)=\frac{\chi_0}{v^k}$ ($\chi_0>0$, $k\ge 1$) 
were derived 
under some smallness conditions for $\chi_0$ (\cite{Fujie-Winkler-Yokota_pe}). 
On the other hand, also in the parabolic-parabolic case, it was shown that 
some smallness condition for $\chi$ leads to global existence and boundedness; 
in the case that $\chi(v)=\frac{\chi_0}{v}$ ($\chi_0>0$) 
Winkler \cite{Winkler_2011} obtained global existence of classical solutions 
under the condition that $\chi_0<\sqrt{\frac{2}{n}}$ 
and Fujie \cite{Fujie_2015} established boundedness of these solutions;  
moreover, Lankeit \cite{Johannes_2016} improved these results 
in the 2-dimensional setting; 
in the case that $\chi(v)\le \frac{\chi_0}{(a+v)^k}$ ($\chi_0>0$, $a\ge0$, $k\ge 1$) 
some smallness condition for $\chi_0$ yields  
global existence and boundedness (\cite{Mizukami-Yokota_02}).  
In the case that $\chi$ is a more general sensitivity, 
Fujie--Senba \cite{F-S;p-e} first established 
global existence and boundedness 
in the two-dimensional parabolic-elliptic system, 
and then they also showed 
existence of radially symmetric bounded solutions to 
the parabolic-parabolic system 
in a two-dimensional ball under the condition that 
$\lambda$ is sufficiently small (\cite{F-S;p-p}). 
Recently, in the nonradial setting, 
a sufficient condition of sensitivity functions for global existence and boundedness in 
the parabolic-parabolic system 
was studied by Fujie--Senba \cite{Fujie-Senba_sensitivity}.  

\smallskip
In summary parabolic-elliptic chemotaxis systems 
often gave us some guide to 
how we could deal with parabolic-parabolic chemotaxis systems;  
however, there have not been rich results on the relation 
between the both systems. 
Namely, it still remains to analyze on the following question except some cases: 
\begin{center}
{\it Does a solution of the parabolic-parabolic system converge to\\ 
that of the parabolic-elliptic problem as $\lambda \searrow 0$?} 
\end{center}
If we can obtain some positive answer to this question, 
then we can see that solutions of both systems have some similar properties; 
thus an answer will enable us to establish 
approaches to obtain properties 
for solutions of the chemotaxis systems.  %
Here, in the case that $\Omega$ is the whole space $\mathbb{R}^n$,  
there are some positive answers to this question
in 2-dimensional case (\cite{Raczynski_2009}) and 
$n$-dimensional case (\cite{Lemarie_2013}). 
Moreover, in the case that $\Omega$ is a bounded domain and 
$\chi(v) \le \frac{\chi_0}{(a+v)^k}$ ($\chi_0>0$, $a \ge 0$, $k>1$), 
a positive answer to this question is also shown    
under the condition that $\chi_0$ is small \cite{Mizukami_fast_sensi}. 
Therefore we can expect a positive answer to this question also in the Keller--Segel system in a bounded domain $\Omega$ in some case. 
The purpose of this paper is to give some positive answer to this question. 

%
%

\smallskip
In order to attain this purpose, this paper investigates the fast signal diffusion limit, which namely is 
convergence of a solution for 
the parabolic-parabolic Keller--Segel system 
\begin{equation}\label{cp}
  \begin{cases}
    (\utau)_t = \Delta \utau 
    - \chi \nabla \cdot (\utau \nabla \vtau),
    & x\in\Omega,\ t>0, 
\\[1mm]
    \lambda (\vtau)_t = \Delta \vtau - \vtau + \utau, 
    & x\in\Omega,\ t>0, 
\\[1mm]
    \nabla \utau\cdot \nu = 
    \nabla \vtau\cdot \nu = 0, 
    & x\in\pa \Omega,\ t>0,
\\[1mm]
    \utau(x,0)=\uinit(x),\ \vtau(x,0)=\vinit(x),
    & x\in\Omega
  \end{cases}
\end{equation}
to that of the parabolic-elliptic Keller--Segel system  
\begin{equation}\label{cpp-e}
  \begin{cases}
    u_t = \Delta u 
    - \chi \nabla \cdot (u \nabla v),
    & x\in\Omega,\ t>0, 
\\[1mm]
    0 = \Delta v - v + u, 
    & x\in\Omega,\ t>0, 
\\[1mm]
    \nabla u \cdot \nu = 
    \nabla v\cdot \nu = 0, 
    & x\in\pa \Omega,\ t>0,
\\[1mm]
    u(x,0)=\uinit(x),
    & x\in\Omega 
  \end{cases}
\end{equation}
as $\lambda \searrow 0$, 
where $\Omega$ is a bounded domain in $\Rn$ ($n\ge 2$) 
with smooth boundary $\pa \Omega$ 
and $\nu$ is the outward normal vector to $\pa\Omega$; 
$\chi,\lambda > 0$ is a constant; 
the initial functions $\uinit,\vinit$ are assumed to be nonnegative functions. 
The unknown functions $\utau $ and $u$   
represent the population density of the species and 
$\vtau$ and $v$ show the concentration of the 
chemical substance 
at place $x$ and time $t$. 

%
%
%
%
Now the main results 
read as follows. 
The first theorem is concerned with 
global existence and the fast signal diffusion limit of solutions for 
the higher-dimensional Keller--Segel system under smallness conditions for 
the initial data. 
\begin{thm}\label{mainth} 
Let $\Omega$ be a bounded domain in $\Rn$ $(n\ge 3)$ with smooth boundary 
and let $\chi>0$ be a constant. 
Assume that 
$\uinit$ and $\vinit$ satisfy  
%
\begin{align}\label{ini} 
0\leq \uinit\in C(\ol{\Omega}), \quad
0\leq \vinit \in W^{1,q}(\Omega) 
\end{align} 
with some $q>n$. 
Then for all $p>\frac n2$ there exists $\ep_0=\ep_0(p,q,\chi,|\Omega|)>0$ such that, 
if $\uinit$ and $\vinit$ satisfy 
\begin{align}\label{condi;HD}
  \lp{p}{\uinit} < \ep_0 
\quad 
  \mbox{and} 
\quad
  \lp{q}{\nabla \vinit} < \ep_0, 
\end{align}
then for all $\lambda > 0$ 
the problem \eqref{cp} possesses 
a unique global bounded solution $(\utau,\vtau)$ 
which is a pair of nonnegative functions 
\[
  \utau,\vtau \in 
  C(\overline{\Omega}\times [0,\infty))
  \cap 
  C^{2,1}(\overline{\Omega}\times (0,\infty)). 
\]
Moreover, if $\uinit$ and $\vinit$ satisfy \eqref{condi;HD}, 
then there are unique functions 
\[
  u\in C(\obar\times [0,\infty))\cap 
  C^{2,1}(\obar\times (0,\infty)) 
\ \mbox{and} \ 
  v \in C^{2,0}(\obar\times (0,\infty))\cap 
  L^\infty(0,\infty;W^{1,q}(\Omega)) 
\]
such that  
the solution $(\utau,\vtau)$ of \eqref{cp} satisfies 
\begin{align*}
 &\utau \to u \quad 
 \mbox{in}\ C_{\rm loc}(\obar\times [0,\infty)),
\\
 &\vtau \to v \quad \mbox{in}\ 
 C_{\rm loc}(\obar\times (0,\infty)) \cap L^2_{\rm loc}((0,\infty);W^{1,2}(\Omega)) 
\end{align*}
as $\lambda\searrow 0$, 
and the pair of the functions $(u,v)$ solves \eqref{cpp-e} classically. 
\end{thm}
%
%

As an application of this result, 
we can establish a new result which provides global existence and boundedness 
in the higher-dimensional parabolic-elliptic Keller--Segel system \eqref{cpp-e} 
under some smallness condition for initial data $\uinit$. 

\begin{corollary}\label{maincoro1}
Let $\Omega$ be a bounded domain in $\Rn$ $(n\ge 3)$ with smooth boundary 
and let $\chi>0$ be a constant. 
Then for all $p>\frac{n}{2}$ there exists $\ep_1=\ep(p,\chi,|\Omega|) >0$ such that, 
if $\uinit\in C(\ol{\Omega})$ satisfies that 
\begin{align*}
  \lp{p}{\uinit} < \ep_1 
\end{align*}
holds, %
then the problem \eqref{cpp-e} possesses a 
unique global bounded classical solution.   
\end{corollary} 

\begin{remark}
In these results we assume the smallness conditions for  
$\lp{p}{\uinit}$ and $\lp{q}{\nabla \vinit}$ with some $p> \frac n2$ and $q>n$, 
instead of $p=\frac{n}{2}$ and $q=n$ 
which are the conditions assumed in \cite{Xinru_higher}; 
we could not attain fast signal diffusion limit under the smallness 
conditions in optimal spaces. 
\end{remark}

In the 2-dimensional setting, 
it is known that global existence and boundedness in \eqref{cp} hold 
under the condition that 
$\lp{1}{\uinit}\le \frac{4\pi}{\chi}$ (\cite{Nagai-Senba-Yoshida}). 
Thanks to this previous work, we attain the fast signal diffusion limit in 
the 2-dimensional Keller--Segel system under the smallness conditions for 
the initial data in the optimal space. 

\begin{thm}\label{mainth2}
Let $\Omega$ be a bounded domain in $\mathbb{R}^2$ with smooth boundary 
and let $\chi>0$ be a constant. 
Assume that 
$\uinit \in C(\overline{\Omega})$ satisfies  
$
\lp{1}{\uinit}<\frac{4\pi}{\chi}. 
$ 
Then there exist unique functions 
\[
  u\in C(\obar\times [0,\infty))\cap 
  C^{2,1}(\obar\times (0,\infty)) 
\ \mbox{and} \ 
  v \in C^{2,0}(\obar\times (0,\infty))\cap 
  L^\infty(0,\infty ;W^{1,q}(\Omega)) 
\]
such that for all $\vinit\in W^{1,q}(\Omega)$ $(q>2)$ 
the global bounded classical solution $(\utau,\vtau)$ of \eqref{cp} satisfies 
\begin{align*}
 &\utau \to u \quad 
 \mbox{in}\ C_{\rm loc}(\obar\times [0,\infty)),
\\
 &\vtau \to v \quad \mbox{in}\ 
 C_{\rm loc}(\obar\times (0,\infty)) \cap L^2_{\rm loc}((0,\infty);W^{1,2}(\Omega)) 
\end{align*}
as $\lambda\searrow 0$, 
and the pair of the functions $(u,v)$ solves \eqref{cpp-e} classically. 
\end{thm}

This result tells us 
a new method to obtain global existence and boundedness in 
the 2-dimensional parabolic-elliptic Keller--Segel system \eqref{cpp-e}. 

\begin{corollary}\label{maincoro2} 
Let $\Omega$ be a bounded domain in $\mathbb{R}^2$ with smooth boundary 
and let $\chi>0$ be a constant. 
If $\uinit\in C(\ol{\Omega})$ satisfies 
$
  \lp{p}{\uinit} < \frac{4\pi}{\chi},  
$ 
then the problem \eqref{cpp-e} possesses a 
unique global bounded classical solution.   
\end{corollary} 

In the proof of these main results difficulties are caused by the facts that 
$\vtau$ satisfies a parabolic equation and 
$v$ satisfies an elliptic equation.  
Thus we cannot use methods only for parabolic equations and 
only for elliptic equations when we would like to obtain some 
error estimate for solutions of \eqref{cp} and those of \eqref{cpp-e}, 
and it seems to be difficult to combine these methods. 
Therefore we rely on a compactness method to obtain 
convergence of a solution $(\utau,\vtau)$ as $\lambda\searrow 0$, 
which is the same strategy as that of the proof of \cite[Theorem 1.3]{Mizukami_fast_sensi}. 
In order to use a compactness method 
some estimate for the solution uniformly in time and $\lambda$ is required. 
In the chemotaxis system with signal dependent sensitivity 
the boundedness of $\into \utau^p\cd \exp\{-r \int_0^{\vtau \cd} \chi(s)\,ds\}$ 
with some $r>0$ 
and the fact $\int_0^\infty \chi(s)\, ds < \infty$ lead to the desired estimate 
(\cite{Mizukami_fast_sensi}). 
Nevertheless, in the Keller--Segel setting, it is difficult to obtain 
the boundedness of $\int_0^{\vtau\cd} \chi \,ds = \chi \vtau \cd$. 
Thus we should give the other method to obtain the desired estimate in the Keller--Segel setting.  
However, 
in the higher-dimensional case, 
a construction of some estimate for the solution 
uniformly in time and $\lambda$ is a 	challenging problem: 
Indeed, in the previous works \cite{Xinru_higher,win_aggregationvs} the following inequality 
was 
obtained: 
\[
 \lp{\infty}{\utau \cd -e^{t\Delta} \uinit} \le C(1 + t^{-\alpha}) e^{-\beta t} 
 \quad \mbox{for all} \ t>0
\]
with some $C,\alpha,\beta>0$, which could not lead to 
the uniform-in-time estimate for the solution. 
This is one of the reason why we could not attain fast signal diffusion limit 
under the smallness conditions in optimal spaces in the higher-dimensional setting. 
To establish the $L^\infty$-estimate for $\utau$ uniformly in time and $\lambda$ 
we modified the method 
in \cite{win_aggregationvs}. Let $\ep>0$ be a constant fixed later and put 
\[
 \tlam := \sup \left\{\hat{T}>0\, {\big |}\, \lp{\theta}{\utau\cd - e^{t\Delta}\uinit }< \ep \quad  
 \mbox{for all} \ t\in (0,\hat{T}) \ \mbox{and all} \ \lambda > 0 \right\} \le \infty  
\] 
with some $\theta > \frac n2$, 
which is different from a setting in \cite{win_aggregationvs}.  
Then, under the conditions that $\lp{p}{\uinit} \le \ep$ and $\lp{q}{\nabla \vinit} \le \ep$, 
we can see that 
\[
  \lp{\theta }{\utau\cd - e^{t\Delta}\uinit} \le C(\ep) \ep 
\quad 
  \mbox{for all} \ t \in (0,\tlam) \ \mbox{and all} \ \lambda>0,  
\] 
where $C(\ep)>0$ is a constant such that $C(\ep) \searrow 0$ 
as $\ep \searrow 0$. 
Thus by choosing $\ep >0$ satisfying $C(\ep) < 1$, 
we can obtain the $L^{\theta}$-estimate for $\utau$ 
uniformly in time and $\lambda$, 
which with the standard $L^p$-$L^q$ estimate for 
the Neumann heat semigroup on bounded domains 
implies the desired estimate for $\utau$. 
This strategy enables us to pass to the fast signal diffusion limit; 
however, it 
also lets us assume that $\lp{p}{\uinit}$ and $\lp{q}{\nabla \vinit}$ are small with some 
$p>\frac n2$ and $q>n$ in Theorem \ref{mainth}. 
On the other hand, in the 2-dimensional setting, 
by using a combination of an argument in the proof of 
\cite[Theorem 1.1]{Nagai-Senba-Yoshida} and a compactness method 
we can show fast signal diffusion limit under the smallness conditions for 
the initial data in optimal spaces. 

%
%
This paper is organized as follows. 
In Section 2 we collect basic facts 
which will be used later. 
In Section 3  
we prove global existence and 
uniform-in-$\lambda$ boundedness in \eqref{cp}; 
we divide the section into Sections 3.1 and 3.2 according to the higher-dimensional setting 
and the $2$-dimensional setting, respectively.  
Section 4 is devoted to the proofs of the main results 
according to arguments in \cite{Wang-Winkler-Xiang}; 
we show convergence of the solution $(\utau,\vtau)$ for \eqref{cp} 
as $\lambda\searrow 0$ 
by using the uniform-in-$\lambda$ estimate 
established in Section 3. 

%
%

\section{Preliminaries}

In this section we collect results which will be used later. 
We first recall the well-known result concerned with local existence of solutions to \eqref{cp} 
(see e.g., \cite[Lemma 3.1]{B-B-T-W}). 

%
%
%
%
\begin{lem}\label{lem;localexistence}
Let $\Omega$ be a bounded domain in $\Rn$ $(n\ge 2)$ 
with smooth boundary, 
and let $\chi>0$ be a constant. 
Then for all $\lambda>0$ and any $\uinit,\vinit$ satisfying \eqref{ini} 
there exists $\tmax\in (0,\infty]$ such that 
the problem \eqref{cp} possesses a unique solution 
$(\utau,\vtau)$ fulfilling 
\begin{align*}
  &\utau\in C(\overline{\Omega}\times [0,\tmax))
  \cap C^{2,1}(\overline{\Omega}\times (0,\tmax)), 
\\
  &\vtau\in C(\overline{\Omega}\times [0,\tmax))
  \cap C^{2,1}(\overline{\Omega}\times (0,\tmax)) 
  \cap L^\infty_{\rm loc}([0,\tmax);W^{1,q}(\Omega)),
\\
  &\utau (x,t)\ge 0\quad \mbox{and} \quad \vtau(x,t) \ge 0 \quad 
  \mbox{for all}\ x\in \Omega, \ t>0 \ \mbox{and all} \ \lambda>0,
\\
  &\into \utau\cd = \into \uinit
  \quad \mbox{for all}\ t\in (0,\tmax) \ \mbox{and all}\ \lambda>0.
\end{align*}
Moreover, either $\tmax=\infty$ or 
\begin{align*}
  \limsup_{t\to\tmax}(\lp{\infty}{\utau\cd}
  + \|\vtau\cd\|_{W^{1,q}(\Omega)})=\infty. 
\end{align*}
\end{lem}
%

%
%
%
%
We next introduced the $L^p$-$L^q$ estimates for 
the Neumann heat semigroup on bounded domains 
which are often utilized to estimate terms coming from 
the variation-of-constants representation for the solutions. 
The following lemma and its proof can be found in 
\cite[Lemma 2.1]{Xinru_higher} 
(or see \cite[Lemma 1.3]{win_aggregationvs}). 

\begin{lem}\label{lem;estif;semi}
Let $(e^{t\Delta})_{t\ge 0}$ be the Neumann heat semigroup in $\Omega$, 
and let $\alpha>0$ denote the first nonzero eigenvalue 
of $-\Delta$ in $\Omega$ under Neumann boundary conditions. 
There are $k_1,k_2,k_3,k_4>0$ only depending on $|\Omega|$ satisfying the following properties. 
 \begin{itemize}
  \item[{{\rm (i)}}] 
    If $1\le q \le p \le \infty$, then 
    \[
      \lp{p}{e^{t\Delta} \varphi} \le 
        k_1 (1+ t^{\frac{n}{2}(\frac{1}{q}-\frac{1}{p})})e^{-\alpha t} \lp{q}{\varphi} 
    \quad \mbox{for all} \ t>0
    \] 
    holds for all $\varphi \in L^q(\Omega)$ satisfying $\into \varphi = 0$. 
  \item[{{\rm (ii)}}]
    If $1\le q\le p \le \infty$, then 
    \[
     \lp{p}{\nabla e^{t\Delta}\varphi}\le 
     k_2 (1 + t^{-\frac 12-\frac n2(\frac 1q -\frac 1p)})
     e^{-\alpha t}\lp{q}{\varphi} 
         \quad \mbox{for all} \ t>0
    \]
    is true for all $\varphi \in L^q(\Omega)$. 
  \item[{{\rm (iii)}}]
    If $2\le q\le p \le \infty$, then 
    \[
      \lp{p}{\nabla e^{t\Delta} \varphi} \le 
      k_3 (1+t^{-\frac{n}{2}(\frac 1q -\frac 1p)})e^{-\alpha t}\lp{q}{\nabla \varphi}
    \quad \mbox{for all} \ t>0
    \] 
    is valid for all $\varphi \in W^{1,p}(\Omega)$. 
  \item[{{\rm (iv)}}]
    If $1<q\le p \le \infty$, then 
    \[
      \lp{p}{e^{t\Delta }\nabla \cdot \varphi} \le 
      k_4 (1 + t^{-\frac 12 -\frac n2 (\frac 1q - \frac 1p)})e^{-\alpha t}
      \lp{q}{\varphi}
    \quad \mbox{for all} \ t>0
    \] 
    holds for all $\varphi\in (W^{1,p}(\Omega))^n$. 
 \end{itemize} 
\end{lem}

%
%
%
%

We finally give the following result which plays an important role 
in obtaining uniform-in-$\lambda$ boundedness of solutions to \eqref{cp}. 
\begin{lem}\label{lem;estimateforufromp}
Let $\lambda>0$. 
If there exist $p>\frac{n}{2}$ and $M>0$ 
such that  
\begin{align*}
\lp{p}{\utau\cd} \le M
\quad \mbox{for all}\ t\in (0,\tmax), 
\end{align*}
then there exists $C=C(p,M)>0$ such that 
\begin{align*}
\lp{\infty}{\utau\cd} + \wmp{1,q}{\vtau\cd }\le C
\quad \mbox{for all} \ t\in (0,\tmax). 
\end{align*}
Moreover, if $p$ and $M$ are independent of $\lambda\in (0,\lambda_0)$ 
with some $\lambdaz >0$, 
then $C$ is also independent of $\lambda\in (0,\lambda_0)$. 
\end{lem}
\begin{proof}
The proof is a combination of \cite[Lemmas 2.3 and 2.4]{Mizukami_fast_sensi} 
(the proof is based on an application of the $L^p$-$L^q$ estimates for 
the Neumann heat semigroup in the proof of \cite[Lemma 3.2]{B-B-T-W}). 
\end{proof}

%
%

\section{Uniform-in-$\lambda$ boundedness}

In this section we establish global existence of solutions to \eqref{cp} 
and their uniform-in-$\lambda$ boundedness. 

\subsection{The higher-dimensional setting}
In this subsection we will deal with 
the higher-dimensional Keller--Segel system \eqref{cp}. 
Aided by Lemma \ref{lem;estimateforufromp}, 
we shall only verify the $L^{p_0}$-estimate for $\utau$ with some $p_0>\frac n2$. 
We first prove the following lemma 
which enables us to pick appropriate constants 
in the proof of Lemma \ref{lem;Lpesti;HD}. 

%
%
%
%
\begin{lem}\label{lem;constant}
Let $p>\frac n2$ and $q>n$. 
Then there are constants $\theta,q_0,\mu>0$ such that 
\begin{align*}
 &\theta \in 
   I_1 := \left(
     p, \min\left\{ 
         \frac{npq}{(np+nq-pq)_+}, 
         \frac{np}{2(n-p)_+}
       \right\}
   \right), 
\\
 &q_0 \in 
   I_2 := \left(
     \max\left\{
       1, \frac{np\theta}{p\theta +np -n\theta}
       \right\}, 
     \min\left\{
       q, \frac{np}{(n-p)_+}
       \right\}
   \right),
\\
 &\mu \in 
   I_3 := \left(
     \max\left\{ 
       1, 
       \frac{n\theta}{n+\theta}, 
       \frac{np\theta}{p\theta + 2np -n\theta} 
     \right\}, 
   \min\left\{
       q_0, 
       \frac{q_0 \theta}{q_0 + \theta}
     \right\}
   \right). 
\end{align*}
\end{lem}
\begin{proof}
Since we have from the conditions 
$q>n$ and $p>\frac n2$ that 
\[
  (np+nq-pq) p < npq 
\quad \mbox{and} \quad 
  2(n-p) < n, 
\]
we can verify that  
\[
I_1 = \left(
     p, \min\left\{ 
         \frac{npq}{(np+nq-pq)_+}, 
         \frac{np}{2(n-p)_+}
       \right\}
   \right)
   \neq \emptyset. 
\]
Thus we can take $\theta \in I_1$. We next see that 
\[
I_2 =  \left(
     \max\left\{
       1, \frac{np\theta}{p\theta +np -n\theta}
       \right\}, 
     \min\left\{
       q, \frac{np}{(n-p)_+}
       \right\}
   \right)\neq \emptyset. 
\]
Noticing from the fact $\theta \in I_1 \subset (p,\frac{np}{(n-p)_+})$ that 
\[
  p\theta + np -n\theta >0 
\quad \mbox{and} \quad   
  \frac{np\theta}{p\theta + np - n\theta} \ge 1, 
\]
we will only confirm that 
\begin{align}\label{check;I_2}
\frac{np\theta}{p\theta +np -n\theta} < \min\left\{
       q, \frac{np}{(n-p)_+}
       \right\}
\end{align} 
holds. 
Here since the fact $\theta < \min\{\frac{npq}{(np+nq-pq)_+}, \frac{np}{2(n-p)_+}\}$ 
implies 
\[
  np\theta < q (p\theta + np -n\theta) 
\quad \mbox{and} \quad 
  \theta (n-p) < p\theta + np - n\theta, 
\]
we can verify that \eqref{check;I_2} is true, 
which tells us that $I_2\neq \emptyset$. 
Therefore we can choose $q_0 \in I_2$. 
We finally confirm that 
\[
I_3 = \left(
     \max\left\{ 
       1, 
       \frac{n\theta}{n+\theta}, 
       \frac{np\theta}{p\theta + 2np -n\theta} 
     \right\}, 
   \min\left\{
       q_0, 
       \frac{q_0 \theta}{q_0 + \theta}
     \right\}
   \right) \neq \emptyset.  
\]
Here we note that $p\theta + 2np -n\theta > p\theta + np -n\theta > 0$ 
and $\frac{q_0 \theta}{q_0+\theta} < q_0$. 
Since the facts $\theta >p >\frac n2 \ge \frac{n}{n-1}$ ($n\ge 3$) and 
$p\theta + 2np -n\theta < p(n+\theta)$ 
derive that 
\[
  1\le \frac{n\theta}{n+\theta} 
  < \frac{np\theta}{p\theta + 2np -n\theta}, 
\]
we shall only see that 
\begin{align}\label{check;I_3}
 \frac{np\theta}{p\theta + 2np -n\theta} < \frac{q_0 \theta}{q_0+\theta}. 
\end{align}
Now aided by the relation $\frac{np\theta}{p\theta + np -n\theta} < q_0$, 
we establish that 
\[
  np(\theta+q_0) < q_0 (p\theta + 2np -n\theta)
\]
holds. Therefore we have \eqref{check;I_3}, which enables us to find 
a constant $\mu \in I_3$. 
This completes the proof. 
\end{proof}

%
%
%
%

Then we can show the following lemma which entails the desired estimate for $\utau$. 
\begin{lem}\label{lem;Lpesti;HD}
Let $p>\frac n2$. 
Then there exists a positive constant 
$\ep_0 = \ep_0(p,q,\chi,|\Omega|)$ 
such that, if $\uinit$  and $\vinit$ satisfy 
  \[
    \lp{p}{\uinit} < \ep_0 
  \quad \mbox{and} \quad 
    \lp{q}{\nabla \vinit} < \ep_0, 
  \]
then there exist $p_0>\frac n2$ and 
$C>0$ which are independent of $\lambda>0$ such that 
  \[
    \lp{p_0}{\utau\cd} \le C 
  \]
for all $t\in (0,\tmax)$ and all $\lambda>0$. 
\end{lem}
\begin{proof}
Let $\ep > 0$ be a constant fixed later, and assume that 
$\uinit$ and $\vinit$ satisfy 
\begin{align}\label{condi;HD;lem}
  \lp{p}{\uinit} \le \ep 
\quad  \mbox{and} \quad 
  \lp{q}{\nabla \vinit} \le \ep 
\end{align}
with some $p > \frac n2$. 
Then invoking to Lemma \ref{lem;constant}, 
we can take $\theta, q_0, \mu \ge 1$ such that 
\[
  \theta \in I_1, \quad q_0 \in I_2 \quad \mbox{and} \quad \mu \in I_3, 
\]
where $I_1,I_2,I_3$ are intervals defined in Lemma \ref{lem;constant}. 
Now we put 
\begin{align*}
 \tlam := \sup \left\{\hat{T}\in (0,\tmax)\, {\big |}\, \lp{\theta}{\utau\cd - e^{t\Delta}\uinit }< \ep \quad   
 \mbox{for all} \ t\in (0,\hat{T})\right\}.  
\end{align*}
Then because $\utau (\cdot, 0) - e^{0\cdot \Delta} \uinit = 0$ and 
the function $t\mapsto \utau \cd - e^{t\Delta} \uinit$ is continuous 
on $[0,\tmax)$, 
$\tlam$ is well-defined and positive with $\tlam \le \tmax$. 
We first note from the standard $L^p$-$L^q$ estimate for the Neumann heat semigroup 
that there is $C_1 = C_1(|\Omega|) >0$ 
such that for all $r\in [1,\theta)$, 
\begin{align}\notag
  \lp{r}{\utau\cd} 
  & \le \lp{r}{\utau\cd - e^{t\Delta} \uinit} + \lp{r}{e^{t\Delta}\uinit} 
\\\notag
  &\le |\Omega|^{\frac 1r -\frac1\theta} \lp{\theta}{\utau\cd - e^{t\Delta} \uinit} + 
     C_1 \lp{r}{\uinit}
\\\label{esti;Lr;u}
  &\le  |\Omega|^{\frac 1r -\frac1\theta} \ep + C_1 \lp{r}{\uinit}
\end{align}
for all $t\in (0,\tlam)$, 
which with the relation $p<\theta$ and  
\eqref{condi;HD;lem} tells us that 
\[
  \lp{p}{\utau\cd} \le 
  (|\Omega|^{\frac 1p -\frac1\theta} + C_1) \ep 
\]
for all $t\in (0,\tlam)$. 
We then obtain from the variation-of-constants representation for $\vtau$, 
the fact $q_0 < q$ and Lemma \ref{lem;estif;semi} (ii), (iii) that 
\begin{align*}
  \lp{q_0}{\nabla \vtau\cd } &\le 
  \lp{q_0}{\nabla e^{\frac t\lambda (\Delta -1)}\vinit} 
  + \frac 1\lambda 
      \int_0^t \lp{q_0}{\nabla e^{\frac{t-s}{\lambda} (\Delta-1)} \utau(\cdot,s)}\,ds  
\\
  &\le C_2 \lp{q_0}{\nabla \vinit} 
    + \frac{C_3 \ep}{\lambda} \int_0^t \left( 
      1+\left(\frac{t-s}{\lambda}\right)^{-\frac 12 -\frac n2(\frac 1p -\frac 1{q_0})}    
    \right) e^{-\alpha (\frac{t-s}{\lambda})}\,ds
\\
  &\le C_2 |\Omega|^{\frac 1{q_0}-\frac 1q}\lp{q}{\nabla \vinit} 
  + C_3 \ep \int_0^{\frac{t}{\lambda}} \left( 
      1+\sigma^{-\frac 12 -\frac n2(\frac 1p -\frac 1{q_0})}    
      \right) e^{-\alpha \sigma}\,d\sigma 
\end{align*}
for all $t\in (0,\tlam)$ with some $C_2  = C_2(|\Omega|)>0$ and $C_3 = C_3(p,q,|\Omega|) > 0$. 
Since the fact $q_0 < \frac{np}{(n-p)_+}$ implies 
$\frac 12 + \frac n2(\frac 1p -\frac 1{q_0})<1$, 
from \eqref{condi;HD;lem} we infer that 
\begin{align*}
  C_2& |\Omega|^{\frac 1{q_0}-\frac 1q}\lp{q}{\nabla \vinit} 
  + C_3 \ep \int_0^{\frac{t}{\lambda}} \left( 
      1+\sigma^{-\frac 12 -\frac n2(\frac 1p -\frac 1{q_0})}    
      \right) e^{-\alpha \sigma}\,d\sigma \le C_4 \ep
\end{align*} 
holds with $C_4:= C_2 |\Omega|^{\frac 1{q_0}-\frac 1q}  
  + C_3 \int_0^{\infty} ( 
        1+\sigma^{-\frac 12 -\frac n2(\frac 1p -\frac 1{q_0})}    
      ) e^{-\alpha \sigma}\,d\sigma < \infty$, which means that 
\begin{align}\label{esti;nablav;Lq0}
  \lp{q_0}{\nabla \vtau \cd} \le C_4 \ep 
\end{align}
for all $t\in (0,\tlam)$. 
Finally, in order to show $\tlam=\tmax$, we will show that 
\[
  \lp{\theta}{\utau\cd -e^{t\Delta} \uinit} \le C\ep
  \quad 
  \mbox{for all} \ t\in (0,\tlam)
\]
with some $C<1$. 
Employing the variation-of-constant formula for $\utau$ and 
Lemma \ref{lem;estif;semi} (iv), we see that
\begin{align}\label{esti;u-uinit;1Step}
  \lp{\theta}{\utau \cd -e^{t\Delta} \uinit} 
\le 
  \chi \int_0^t ( 1 + (t-s)^{-\frac 12-\frac{n}{2}(\frac 1\mu -\frac 1\theta)})
  e^{-\alpha (t-s)} \lp{\mu}{\utau (\cdot,s) \nabla \vtau (\cdot,s)}\, ds
\end{align} 
Now thanks to the facts $\mu<q_0$, $1< \frac{q_0 \mu}{q_0-\mu} < \theta$  and 
\eqref{esti;Lr;u}--\eqref{esti;nablav;Lq0}, 
we derive from the H\"older inequality and the interpolation inequality that 
\begin{align}\label{esti;Lmu;unablav} \notag
  &\lp{\mu}{\utau(\cdot,s) \nabla \vtau(\cdot,s)} 
\\\notag
  &\le 
  \lp{\frac{q_0 \mu}{q_0-\mu}}{\utau (\cdot,s)} \lp{q_0}{\nabla \vtau (\cdot,s)}  
\\ \notag
  &\le 
  \lp{1}{\utau (\cdot,s)}^a\lp{\theta}{\utau (\cdot,s)}^{1-a} \lp{q_0}{\nabla \vtau (\cdot,s)} 
\\ 
  &\le 
  C_5 \ep^{1+a} \left(\lp{\theta}{\utau(\cdot,s)-e^{s\Delta}\uinit}^{1-a} 
    + \lp{\theta}{e^{s\Delta}(\uinit-\ol{\uinit})}^{1-a}
    + \lp{\theta}{e^{s\Delta}\ol{\uinit}}^{1-a}\right) 
\end{align}
with some $C_5=C_5(p,q,|\Omega|) > 0$,
where $a=\frac{(q_0-\mu)\theta -q_0 \mu}{q_0\mu (\theta-1)}\in (0,1)$ 
and $\ol{\uinit}:= \frac{1}{|\Omega|}\into \uinit$. 
Here from the H\"older inequality, the Young inequality 
and Lemma \ref{lem;estif;semi} (i) 
we can find $C_6 =C_6(p,q,|\Omega|)>0$ and 
$C_7 =C_7 (p,q,|\Omega|) >0$ such that 
\begin{align}\label{esti;firstterm}
  \lp{\theta}{e^{s\Delta}\ol{\uinit}}^{1-a} = \lp{\theta}{\ol{\uinit}}^{1-a} 
  \le C_6 \lp{p}{\uinit}^{1-a} \le C_6 \ep^{1-a} 
\end{align}
and 
\begin{align}\label{esti;secondterm}\notag
  \lp{\theta}{e^{s\Delta}(\uinit-\ol{\uinit})}^{1-a} 
  &\le (1-a) \lp{\theta}{e^{s\Delta}(\uinit-\ol{\uinit})} + a 
\\\notag
  &\le C_7 (1+s^{-\frac n2 (\frac 1p - \frac 1\theta)})e^{-\alpha s}
     \lp{p}{\uinit - \ol{\uinit}} 
     +a
\\
  &\le C_7(1+|\Omega|^{-1+\frac 1p})\ep 
     (1+s^{-\frac n2 (\frac 1p - \frac 1\theta)})e^{-\alpha s} +a.    
\end{align}
Since the fact $\frac 12 + \frac n2 (\frac 1\mu -\frac 1\theta) <1$ leads to 
$\int_0^\infty (1+\sigma^{-\frac 12 - \frac n2 (\frac 1\mu -\frac 1\theta)})e^{-\alpha \sigma} \,d\sigma <\infty$ and  the relation 
$
  1-\frac 12 -\frac n2 (\frac 1\mu -\frac 1\theta ) - \frac n2(\frac 1p - \frac 1\theta)>0
$
derives from \cite[Lemma 1.2]{win_aggregationvs} that  
\begin{align*}
  &\int_0^t ( 1 + (t-s)^{-\frac 12-\frac{n}{2}(\frac 1\mu -\frac 1\theta)})e^{-\alpha (t-s)}  
  (1+s^{-\frac n2 (\frac 1p - \frac 1\theta)})e^{-\alpha s}\, ds 
\\
  &\le C_8 (1 + t^{\min\{0, 1-\frac 12 -\frac n2 (\frac 1\mu -\frac 1\theta ) - \frac n2(\frac 1p - \frac 1\theta)\}})e^{-\alpha t}\le 2C_8   
\end{align*}
for all $t >0$ with some $C_8=C_8(p,q)>0$, 
plugging \eqref{esti;firstterm} and \eqref{esti;secondterm} 
into \eqref{esti;u-uinit;1Step} and \eqref{esti;Lmu;unablav} implies that 
\begin{align*}
  \lp{\theta}{\utau \cd -e^{t\Delta} \uinit} 
&\le 
  C_9 \ep^{1+a}\left(
    \sup_{s\in (0,\tlam)} \lp{\theta}{\utau(\cdot,s)-e^{s\Delta}\uinit}^{1-a} + \ep +1 + \ep^{1-a}
  \right)
\\ 
&\le 
 C_{9} \ep^{1+a} (2\ep^{1-a} + \ep +1) 
\end{align*}
for all $t\in (0,\tlam)$ with some $C_9 = C_9(p,q,\chi,|\Omega|)>0$.  
Thus if we take $\ep>0$ satisfying  
\[
  C_{9} \ep^{a} (2\ep^{1-a} + \ep +1) <1, 
\]
then the continuity of 
the function $t\mapsto \lp{\theta}{\utau\cd - e^{t\Delta} \uinit}$ 
concludes that 
\[
  \tlam=\tmax, 
\]
which namely means that 
\begin{align}\label{esti;Lp;u-eunit}
\lp{\theta}{\utau(\cdot,t)-e^{t\Delta}\uinit} \le \ep 
\end{align} 
for all $t\in (0,\tmax)$. 
Here, since $\ep > 0$ is independent of $\lambda >0$, 
we note that \eqref{esti;Lp;u-eunit} holds for all $t \in (0,\tmax)$ and all $\lambda>0$, 
which together with the maximum principle 
\[
  \lp{\infty}{e^{t\Delta} \uinit}\le \lp{\infty}{\uinit}
  \quad \mbox{for all} \ t>0
\] 
enables us to see that 
\begin{align}\label{esti;Ltheta;u}
  \lp{\theta}{\utau\cd} \le \ep + |\Omega|^{\frac{1}{\theta}} \lp{\infty}{\uinit}
\end{align} 
for all $t\in (0,\tmax)$ and all $\lambda >0$. 
Noticing that $\theta > p > \frac{n}{2}$ holds and 
$\theta$, $\ep$ are independent of $\lambda$, 
from \eqref{esti;Ltheta;u} we can attain the goal of the proof.   
\end{proof}

Here we are in the position to 
prove global existence and uniform-in-$\lambda$ 
boundedness in the higher-dimensional Keller--Segel system \eqref{cp}.  

%
%
%
%
\begin{lem}\label{lem;boundedness;HD}
Let $p>\frac{n}{2}$. 
Assume that $\uinit$ and $\vinit$ satisfy 
  \[
    \lp{p}{\uinit} < \ep_0 
  \quad \mbox{and} \quad 
    \lp{q}{\nabla \vinit} < \ep_0, 
  \]
where $\ep_0$ is the constant 
defined in Lemma \ref{lem;Lpesti;HD}. 
Then $\tmax =\infty$ holds, and 
there exists $C>0$ independent of $\lambda>0$ such that 
  \[  
    \lp{\infty}{\utau\cd} \le C
  \]
for all $t\in (0,\infty)$ and all $\lambda > 0$. 
\end{lem}
\begin{proof}
A combination of  Lemmas \ref{lem;estimateforufromp} 
and \ref{lem;Lpesti;HD}, along with the extensibility criterion 
directly leads to this lemma. 
\end{proof}

%
%
\subsection{The 2-dimensional setting} 

In this subsection we will show uniform-in-$\lambda$ boundedness in 
the $2$-dimensional Keller--Segel system. The proof is mainly based on 
arguments in the proof of \cite[Theorem 1.1]{Nagai-Senba-Yoshida}.  
Thus we will only give short proofs.  

\begin{lem}\label{lem;Lpesti;2D}
Assume that $\uinit$ satisfies 
\[
  \lp{1}{\uinit} < \frac{4\pi}{\chi}. 
\]
Then for all $\lambda_0 >0$ there exists $C>0$ such that 
\[
  \lp{2}{\utau\cd} \le C 
\]
for all $t\in (0,\tmax)$ and all $\lambda \in (0,\lambda_0)$. 
\end{lem}
\begin{proof}
Let $\lambda_0 >0$ be an arbitrary constant. 
From straightforward calculations we can verify that the function 
\[
  \wlam := 
    \into \left(
      \utau \log \utau 
      - \chi \utau \vtau 
      + \frac \chi 2 (|\nabla \vtau|^2 + \vtau^2)
    \right)
\]
satisfies 
\begin{align}\label{rela;Laypnov}
  \frac {d\wlam}{dt} +\chi \lambda \into |(\vtau)_t|^2 + 
  \into \utau |\nabla \cdot (\log \utau - \chi \vtau)| = 0 
\end{align}
for all $\lambda\in (0,\lambda_0)$. 
Then by virtue of the Jensen inequality and the Trudinger--Moser inequality, 
the same argument as in 
the proof of \cite[Lemma 3.4]{Nagai-Senba-Yoshida} derives that 
there is $C_1>0$ such that 
\begin{align}\label{esti;L1;uv}
  \into \utau\cd \vtau\cd \le C_1 \quad \mbox{and} \quad 
  |\wlam(t)| \le C_1
\quad 
\end{align}
for all $t\in (0,\tmax)$ and $\lambda \in (0,\lambda_0)$ 
under the condition that $\lp{1}{\uinit} < \frac{4\pi}{\chi}$. 
Thanks to \eqref{esti;L1;uv}, the relation \eqref{rela;Laypnov} implies that 
\begin{align}\label{esti;L1;ulogu}
  \into |\utau\cd \log \utau\cd| \le \max\left\{\wlam(0) + C_1, \frac{1}{e} \right\} 
\end{align}
and 
\begin{align}\label{esti;|W|} 
  \lambda \int_0^t \lp{2}{(\vtau)_t (\cdot,s)}^2\,ds \le \frac{1}{\chi}(|\wlam(0)|+C_1) 
\end{align}
for all $t\in (0,\tmax)$ and $\lambda \in (0,\lambda_0)$. 
Now we shall show the $L^2$-boundedness of $\utau$. 
Multiplying the first equation in \eqref{cp} by $\frac{1}{2}\utau$ 
and integrating it over $\Omega$, we infer from integration by parts that  
\begin{align}\label{def;L2;1Step} \notag
  \frac 12 \frac {d}{dt} \into \utau^2   
  &= - \into |\nabla \utau|^2 +\chi \into \utau \nabla \utau \cdot \nabla \vtau 
\\
  &= - \into |\nabla \utau|^2 - \frac{\chi \lambda}{2} \into \utau^2(\vtau)_t 
   + \frac{\chi}{2} \into \utau^3 - \frac{\chi}{2} \into \utau^2 \vtau. 
\end{align}
Let $\ep>0$ be a constant fixed later. 
Since the Gagliardo--Nirenberg inequality and its application 
(see \cite[Lemma 3.5]{Nagai-Senba-Yoshida}) derive 
\[
  \frac{\chi}{2} \into \utau^3 
  \le \ep \lp{2}{\utau}^2\lp{1}{\utau \log \utau} 
      + C_2(\lp{1}{\utau \log \utau}^3 + \lp{1}{\utau}^2)  
\]
and 
\begin{align*}
  - \frac{\chi\lambda}{2} \into \utau^2(\vtau)_t 
  &\le C_3 \lambda \lp{2}{(\vtau)_t}^2 
         (\lp{2}{\nabla \utau} \lp{2}{\utau} + \lp{2}{\utau}^2)
\\
  &\le \ep \lp{2}{\nabla \utau}^2 + 
  \left( 
    C_4 \lambda^2 \lp{2}{(\vtau)_t}^2 +\frac{1}{4} 
  \right) \lp{2}{\utau}^2
\end{align*}
with some $C_2,C_3,C_4>0$, 
the relation \eqref{def;L2;1Step} with the nonnegativity of $\vtau$ 
tells us that  
\begin{align*}
  \frac{1}{2}&\frac{d}{dt} \into \utau^2 + 
  (1-\ep -\ep \lp{1}{\utau \log \utau})\into |\nabla \utau|^2 
\\
  &\le  \left( 
    C_4 \lambda^2 \lp{2}{(\vtau)_t}^2 +\frac{1}{4} 
  \right) \into \utau^2 + C_2(\lp{1}{\utau \log \utau}^3 + \lp{1}{\utau}^2). 
\end{align*}
Noticing from the boundedness of $\lp{1}{\utau \log \utau}$ 
(from \eqref{esti;L1;ulogu}) 
that there is $\ep>0$ 
such that 
\[
  1-\ep -\ep \lp{1}{\utau\cd \log \utau\cd} \ge \frac{1}{2} 
\] 
for all $t\in (0,\tmax)$ and $\lambda\in (0,\lambda_0)$,  
we infer from the application of the Gagliardo--Nirenberg inequality 
\[
  \lp{2}{\utau} \le \lp{2}{\nabla \utau}^2 + C_5 \lp{1}{\utau}^2 
\]
with some $C_5>0$ that 
\begin{align}\label{def;L2;2Step}  \notag
  \frac{1}{2}\frac{d}{dt} \into \utau^2 + 
  \frac{1}{2} \left( 
    \frac{1}{2}-C_4 \lambda^2 \lp{2}{(\vtau)_t}^2 
  \right)  
  \into \utau^2 
  &\le   
  C_6(\lp{1}{\utau \log \utau}^3 + \lp{1}{\utau}^2)
  \\
  &\le \frac{1}{2}L 
\end{align}
with some $C_6>0$ and $L>0$. 
Now we put 
\[
  y(t) := \into \utau^2(\cdot,t) 
\quad \mbox{and} \quad 
  \phi (t) := \frac 12 t - \frac{C_4\lambda^2}{2} \int_0^t \lp{2}{(\vtau)_t (\cdot,s)}^2\, ds.  
\]
Then from the differential inequality \eqref{def;L2;2Step}, 
we establish that 
\begin{align*}
  y(t) \le y(0) e^{-\phi(t)} + L e^{-\phi(t)} \int_0^t e^{\phi(s)}\,ds 
  \quad \mbox{for all} \ t\in (0,\tmax) \ \mbox{and} \ \lambda \in (0,\lambda_0).  
\end{align*} 
Thus the boundedness of $\phi(t)$ 
\[
  \frac{1}{2} t - \frac{C_4 \lambda_0}{ 2\chi} (|\wlam(0)| + C_1) \le 
  \phi(t) \le \frac{1}{2} t 
  \quad (t\in (0,\tmax), \ \lambda\in (0,\lambda_0) )
\]
(from \eqref{esti;|W|}) entails that there is $C_7=C_7(\lambda_0)>0$ such that 
\[
    y(t) \le y(0) e^{-\phi(t)} + L e^{-\phi(t)} \int_0^t e^{\phi(s)}\,ds \le C_7   
\] 
for all $t>0$ and $\lambda\in (0,\lambda_0)$, 
which means the end of the proof. 
\end{proof}

%
%
%
%

Thanks to Lemma \ref{lem;Lpesti;2D}, we attain 
global existence and uniform-in-$\lambda\in (0,\lambda_0)$ boundedness of 
the solution $(\utau,\vtau)$ to the $2$-dimensional Keller--Segel system.  

\begin{lem}\label{lem;boundedness;2D}
Assume that $\uinit$ satisfies 
\[
  \lp{1}{\uinit} < \frac{4\pi}{\chi}. 
\]
Then $\tmax =\infty$ holds, and 
for all $\lambda_0>0$ there exists $C>0$ such that 
  \[  
    \lp{\infty}{\utau\cd} \le C
  \]
for all $t\in (0,\infty)$ and all $\lambda\in (0,\lambda_0)$. 
\end{lem}
\begin{proof}
A combination of Lemmas \ref{lem;estimateforufromp} 
and \ref{lem;Lpesti;2D}, along with the extensibility criterion  
leads to this lemma. 
\end{proof}

%
%

\section{Convergence}

In this section we will show that solutions of \eqref{cp} 
converge to those of \eqref{cpp-e}. 
Here we assume that 
there exists a unique global classical solution 
$(\utau,\vtau)$ of \eqref{cp} such that for all $\lambda_0 > 0$ 
there is $C>0$ independent of $\lambda \in (0,\lambda_0)$ such that, 
\[
  \lp{\infty}{\utau\cd} + \wmp{1,q}{\vtau\cd} \le C
\]
for all $t > 0$ and all $\lambda \in (0,\lambda_0)$, 
which is established by Lemmas \ref{lem;boundedness;HD} and 
\ref{lem;boundedness;2D}. 
Arguments in this section are based on those in 
the proof of 
\cite[Theorem 1.1]{Wang-Winkler-Xiang}; 
thus I shall only show brief proofs.   
%
%
%
%
We first confirm the following lemma which is a cornerstone of this work. 
\begin{lem}\label{lem;boundedinholder}
For all sequences of numbers $\{\lambdan\}_{n\in\mathbb{N}}\subset (0,\lambda_0)$ 
satisfying $\lambdan \searrow 0$ as $n\to\infty$ 
there exist a subsequence $\lambdanj \searrow 0$ and functions 
\begin{align*}
  u\in C(\obar\times [0,\infty))\cap 
  C^{2,1}(\obar\times (0,\infty)) 
\ \mbox{and} \ 
  v \in C^{2,0}(\obar\times (0,\infty))\cap 
  L^\infty(0,\infty ;W^{1,q}(\Omega)) 
\end{align*}
such that  
\begin{align*}
 &\utaunj \to u \quad 
 \mbox{in}\ C_{\rm loc}(\obar\times [0,\infty)), 
\\
 &\vtaunj \to v \quad \mbox{in}\ 
 C_{\rm loc}(\obar\times (0,\infty)) \cap L^2_{\rm loc}((0,\infty);W^{1,2}(\Omega))
\end{align*}
as $j\to \infty$. 
Moreover, $(u,v)$ solves \eqref{cpp-e} 
classically. 
\end{lem}
\begin{remark}\label{remark1}
This lemma also gives that 
global existence and boundedness in \eqref{cpp-e} hold 
under the condition that there is a unique global bounded solution in \eqref{cp} 
which is bounded uniformly in $\lambda\in (0,\lambda_0)$.  
\end{remark}
\begin{proof}
From the assumption in this section 
and the standard parabolic regularity argument \cite[Theorem 1.3]{Porzio-Vespri}
we see that $\{\utau\}_{\itau}$ is bounded in 
$C^{\alpha,\frac{\alpha}{2}}_{\rm loc}(\ol{\Omega}\times[0,\infty))$ 
with some $\alpha\in (0,1)$. 
Thus the Arzel\`{a}--Ascoli theorem 
and the boundedness of $\|\nabla \vtau\|_{L^\infty(0,\infty;W^{1,q}(\Omega))}$ yields that we can find a subsequence 
$\lambdanj \searrow 0$ and functions 
\[
  u\in C^{\alpha,\frac{\alpha}{2}}_{\rm loc}
    (\ol{\Omega}\times [0,\infty)) 
\quad \mbox{and} \quad 
  v\in L^\infty (0,\infty;W^{1,q}(\Omega)) 
\]
satisfying 
\begin{align*}
 \utaunj \to u \quad 
 \mbox{in}\ C_{\rm loc}(\obar\times [0,\infty)) 
\quad \mbox{and} \quad 
 \vtaunj \overset{\ast}{\rightharpoonup} v \quad \mbox{in}\ 
 L^\infty (0,\infty;W^{1,q}(\Omega)) 
\end{align*}
as $j\to \infty$. 
Then arguments similar to those in 
the proof of \cite[Theorem 1.1]{Wang-Winkler-Xiang} 
enable us to attain this lemma.  
\end{proof}

%
%
%
%
We next verify the following lemma which implies that 
the pair of functions $(u,v)$ provided 
by Lemma \ref{lem;boundedinholder} is independent of a choice of 
a sequence $\lambda_n\searrow 0$. 

\begin{lem}\label{uniqueness}
A solution $(\uobar,\vobar)$ of \eqref{cpp-e} satisfying 
\[
  \uobar\in C(\obar\times [0,\infty))\cap 
  C^{2,1}(\obar\times (0,\infty)) 
\ \mbox{and} \ 
  \vobar \in C^{2,0}(\obar\times (0,\infty))\cap 
  L^\infty(0,\infty; W^{1,q}(\Omega)) 
\]
is unique. 
\end{lem}
\begin{proof}
Let $({\uobar}_1,{\vobar}_1)$ and $({\uobar}_2,{\vobar}_2)$ be solutions to \eqref{cpp-e}  
and put $y(x,t):= \uobar_1(x,t)-\uobar_2(x,t)$ for $(x,t)\in \Omega \times (0,\infty)$.  
Then aided by the Gronwall-type argument similar to that in 
the proof of \cite[Lemma 2.1]{stinner_tello_winkler},  
we infer that $y(x,t)=0$, which concludes the proof. 
\end{proof}

%
%
%
%
%
%
Finally we shall establish convergence of 
the solution $(\utau,\vtau)$ for \eqref{cp} as $\lambda\searrow 0$. 
\begin{lem}\label{lem;convergence;final}
The solution $(\utau,\vtau)$ of \eqref{cp} with $\lambda\in (0,\lambda_0)$ 
satisfies that  
\begin{align*}
 &\utau \to u \quad 
 \mbox{in}\ C_{\rm loc}(\obar\times [0,\infty)), 
\\
 &\vtau \to v \quad \mbox{in}\ 
 C_{\rm loc}(\obar\times (0,\infty)) \cap L^2_{\rm loc}((0,\infty);W^{1,2}(\Omega)) 
\end{align*}
as $\lambda\searrow 0$, where $(u,v)$ is the solution of \eqref{cpp-e} 
provided by Lemma \ref{lem;boundedinholder}. 
\end{lem}
\begin{proof}
Lemmas \ref{lem;boundedinholder} and \ref{uniqueness} yield that 
there exists the pair of the functions $(u,v)$ such that 
for any sequences $\{\lambda_n\}_{n\in\mathbb{N}}\subset (0,\lambda_0)$ 
satisfying $\lambda_n\searrow 0$ as $n\to \infty$ 
there is a subsequence $\lambda_{n_j} \searrow 0$ 
such that  
\begin{align*}
 &\utaunj \to u \quad 
 \mbox{in}\ C_{\rm loc}(\obar\times [0,\infty)), 
\\
 &\vtaunj \to v \quad \mbox{in}\ 
 C_{\rm loc}(\obar\times (0,\infty)) \cap L^2_{\rm loc}((0,\infty);W^{1,2}(\Omega))
\end{align*}
as $j\to \infty$, which enables us to see this lemma. 
\end{proof}

\begin{proof}[{\rm \bf Proof of Theorem \ref{mainth}}]
Lemmas \ref{lem;boundedness;HD} and \ref{lem;convergence;final} 
directly show Theorem \ref{mainth}. 
\end{proof}

\begin{proof}[{\rm \bf Proof of Corollary \ref{maincoro1}}]
Put $\ep_1=\ep_1(p,\chi,|\Omega|):= \sup_{q\in (n,\infty))}\ep_0(p,q,\chi,|\Omega|)$ and 
let $\uinit$ satisfy $\lp{p}{\uinit} < \ep_1$.  
Then we can pick $q > n$ such that 
$\lp{p}{\uinit} < \ep_0(p,q,\chi,|\Omega|)$. 
Now we choose $\vinit \in W^{1,q} (\Omega)$ satisfying 
$\lp{q}{\nabla \vinit} < \ep_0$. 
By virtue of Theorem \ref{mainth}, 
we can prove Corollary \ref{maincoro1}. 
\end{proof}

\begin{proof}[{\rm \bf Proof of Theorem \ref{mainth2}}]
From Lemmas \ref{lem;boundedness;2D} and \ref{lem;convergence;final} 
we can see Thereom \ref{mainth2}. 
\end{proof}

\begin{proof}[{\rm \bf Proof of Corollary \ref{maincoro2}}]
Theorem \ref{mainth2} directly leads to Corollary \ref{maincoro2}. 
\end{proof}

%
\section*{Acknowledgments}
The author would like to thank Professor Michael Winkler 
for pointing out mistakes in 
the previous arguments in Section 4, 
and moreover, 
appreciated his modifying arguments in his joint paper 
with Professors Yulan Wang and Zhaoyin Xiang 
(\cite{Wang-Winkler-Xiang}). 



\begin{thebibliography}{10}

\bibitem{B-B-T-W}
N.~Bellomo, A.~Bellouquid, Y.~Tao, and M.~Winkler.
\newblock Toward a mathematical theory of {K}eller--{S}egel models of pattern
  formation in biological tissues.
\newblock {\em Math.\ Models Methods Appl.\ Sci.}, 25:1663--1763, 2015.

\bibitem{Biler_1999}
P.~Biler.
\newblock Global solutions to some parabolic-elliptic systems of chemotaxis.
\newblock {\em Adv.\ Math.\ Sci.\ Appl.}, 9:347--359, 1999.

\bibitem{Xinru_higher}
X.~Cao.
\newblock Global bounded solutions of the higher-dimensional {K}eller--{S}egel
  system under smallness conditions in optimal spaces.
\newblock {\em Discrete Contin.\ Dyn.\ Syst.}, 35:1891--1904, 2015.

\bibitem{Childress-Percus_81}
S.~Childress and J.~K. Percus.
\newblock Nonlinear aspects of chemotaxis.
\newblock {\em Math.\ Biosci.}, 56:217--237, 1981.

\bibitem{Fujie_2015}
K.~Fujie.
\newblock Boundedness in a fully parabolic chemotaxis system with singular
  sensitivity.
\newblock {\em J. Math.\ Anal.\ Appl.}, 424:675--684, 2015.

\bibitem{F-S;p-e}
K.~Fujie and T.~Senba.
\newblock Global existence and boundedness in a parabolic-elliptic
  {K}eller--{S}egel system with general sensitivity.
\newblock {\em Discrete Contin.\ Dyn.\ Syst.\ Ser.\ B}, 21:81--102, 2016.

\bibitem{F-S;p-p}
K.~Fujie and T.~Senba.
\newblock Global existence and boundedness of radial solutions to a two
  dimensional fully parabolic chemotaxis system with general sensitivity.
\newblock {\em Nonlinearity}, 29:2417--2450, 2016.

\bibitem{Fujie-Senba_sensitivity}
K.~Fujie and T.~Senba.
\newblock A sufficient condition of sensitivity functions for boundedness of
  solutions to a parabolic-parabolic chemotaxis system.
\newblock preprint.

\bibitem{Fujie-Winkler-Yokota_pe}
K.~Fujie, M.~Winkler, and T.~Yokota.
\newblock Boundedness of solutions to parabolic-elliptic {K}eller--{S}egel
  systems with signal-dependent sensitivity.
\newblock {\em Math.\ Methods Appl.\ Sci.}, 38:1212--1224, 2015.

\bibitem{Herrero-Velazques_1997}
M.~A. Herrero and J.~J.~L. Vel{\'{a}}zquez.
\newblock A blow-up mechanism for a chemotaxis model.
\newblock {\em Ann.\ Scuola Norm.\ Sup.\ Pisa Cl.\ Sci.}, 24:633--683, 1997.

\bibitem{Horstmann-Wang}
D.~Horstmann and G.~Wang.
\newblock Blow-up in a chemotaxis model without symmetry assumptions.
\newblock {\em Eur.\ J.\ Appl.\ Math.}, 12:159--177, 2001.

\bibitem{K-S}
E.~F. Keller and L.~A. Segel.
\newblock Initiation of slime mold aggregation viewed as an instability.
\newblock {\em J. Theor.\ Biol.}, 26:399--415, 1970.

\bibitem{K-S_2}
E.~F. Keller and L.~A. Segel.
\newblock Traveling bands of chemotactic bacteria: {A} theoretical analysis.
\newblock {\em J. Theor.\ Biol.}, 30:235--248, 1971.


\bibitem{Johannes_2016}
J.~Lankeit.
\newblock A new approach toward boundedness in a two-dimensional parabolic
  chemotaxis system with singular sensitivity.
\newblock {\em Math.\ Methods Appl.\ Sci.}, 39:394--404, 2016.

\bibitem{Lemarie_2013}
P.~G. Lemari{\'{e}}-Rieusset.
\newblock Small data in an optimal {B}anach space for the parabolic-parabolic
  and parabolic-elliptic {K}eller--{S}egel equations in the whole space.
\newblock {\em Adv.\ Differ.\ Equ.}, 18:1189--1208, 2013.

\bibitem{Mizoguchi-Souplet_2014}
N.~Mizoguchi and P.~Souplet.
\newblock Nondegeneracy of blow-up points for the parabolic {K}eller--{S}egel
  system.
\newblock {\em Ann.\ Inst.\ H.\ Poincar{\'{e}} Anal.\ Non Lin{\'{e}}aire},
  31:851--875, 2014.

\bibitem{Mizoguchi-Winkler}
N.~Mizoguchi and M.~Winkler.
\newblock Blow-up in the two-dimensional parabolic {K}eller--{S}egel system.
\newblock preprint.

\bibitem{Mizukami_fast_sensi}
M.~Mizukami.
\newblock The fast signal diffusion limit in a chemotaxis system with strong
  signal sensitivity.
\newblock {\em Math.\ Nachr.}, to appear. arXiv:1711.01677 [math.AP].

\bibitem{Mizukami-Yokota_02}
M.~Mizukami and T.~Yokota.
\newblock A unified method for boundedness in fully parabolic chemotaxis
  systems with signal-dependent sensitivity.
\newblock {\em Math.\ Nachr.}, 290:2648--2660, 2017.

\bibitem{Nagai_1995}
T.~Nagai.
\newblock Blow-up of radially symmetric solutions to a chemotaxis system.
\newblock {\em Adv.\ Math.\ Sci.\ Appl.}, 5:581--601, 1995.

\bibitem{Nagai_2001}
T.~Nagai.
\newblock Blowup of nonradial solutions to parabolic-elliptic systems modeling
  chemotaxis in two-dimensional domains.
\newblock {\em J.\ Inequal.\ Appl.}, 6:37--55, 2001.

\bibitem{Nagai-Senba_1998}
T.~Nagai and T.~Senba.
\newblock Global existence and blow-up of radial solutions to a
  parabolic-elliptic system of chemotaxis.
\newblock {\em Adv.\ Math.\ Sci.\ Appl.}, 8:145--156, 1998.

\bibitem{Nagai-Senba-Suzuki_2000}
T.~Nagai, T.~Senba, and T.~Suzuki.
\newblock Chemotactic collapse in a parabolic system of mathematical biology.
\newblock {\em Hiroshima Math.\ J.}, 30:463--497, 2000.

\bibitem{Nagai-Senba-Yoshida}
T.~Nagai, T.~Senba, and K.~Yoshida.
\newblock Application of the {T}rudinger-{M}oser inequality to a parabolic
  system of chemotaxis.
\newblock {\em Funkcial.\ Ekvac.}, 40:411--433, 1997.

\bibitem{Nanjundiah_1973}
V.~Nanjundiah.
\newblock Chemotaxis, signal relaying and aggregation morphology.
\newblock {\em J. theor.\ Biol.}, 42:63--105, 1973.

\bibitem{Osaki-Yagi}
K.~Osaki and A.~Yagi.
\newblock Finite dimensional attractor for one-dimensional {K}eller--{S}egel
  equations.
\newblock {\em Funkcial.\ Ekvac.}, 44:441--469, 2001.

\bibitem{Porzio-Vespri}
M.~M. Porzio and V.~Vespri.
\newblock H{\"{o}}lder estimates for local solutions of some doubly nonlinear
  degenerate parabolic equations.
\newblock {\em J. Differential Equations}, 103:146--178, 1993.

\bibitem{Raczynski_2009}
A.~Raczy{\'{n}}ski.
\newblock Stability property of the two-dimensional {K}eller--{S}egel model.
\newblock {\em Asymptotic Anal.}, 61:35--59, 2009.

\bibitem{Senba_2007}
T.~Senba.
\newblock Type {I\hspace{-.1em}I} blowup of solutions to a simplified
  {K}eller--{S}egel system in two dimensional domains.
\newblock {\em Nonlinear Anal.}, 66:1817--1839, 2007.

\bibitem{Senba-Suzuki_Kokyuroku}
T.~Senba and T.~Suzuki.
\newblock Chemotactic collapse in a parabolic-elliptic system of mathematical
  biology. variational problems and related topics.
\newblock {\em S{$\bar{u}$}rikaisekikenky{$\bar{u}$}sho
  K{$\bar{o}$}ky{$\bar{u}$}roku}, 1117:64--97, 1999.

\bibitem{Senba-Suzuki_2001}
T.~Senba and T.~Suzuki.
\newblock Chemotactic collapse in a parabolic-elliptic system of mathematical
  biology.
\newblock {\em Adv.\ Differential Equations}, 6:21--50, 2001.

\bibitem{stinner_tello_winkler}
C.~Stinner, J.~I. Tello, and M.~Winkler.
\newblock Competitive exclusion in a two-species chemotaxis model.
\newblock {\em J. Math.\ Biol.}, 68:1607--1626, 2014.

\bibitem{Wang-Winkler-Xiang}
Y.~Wang, M.~Winkler, and Z.~Xiang. 
\newblock The fast signal diffusion limit in Keller-Segel(-fluid) systems. 
\newblock arXiv:1805.05263 [math.AP].

\bibitem{win_aggregationvs}
M.~Winkler.
\newblock Aggregation vs.\ global diffusive behavior in the higher-dimensional
  {K}eller--{S}egel model.
\newblock {\em J. Differential Equations}, 248:2889--2905, 2010.

\bibitem{Winkler_2011}
M.~Winkler.
\newblock Global solutions in a fully parabolic chemotaxis system with singular
  sensitivity.
\newblock {\em Math.\ Methods Appl.\ Sci.}, 34:176--190, 2011.

\bibitem{Winkler_2013_blowup}
M.~Winkler.
\newblock Finite-time blow-up in the higher-dimensional parabolic-parabolic
  {K}eller--{S}egel system.
\newblock {\em J. Math.\ Pures Appl.}, 100:748--767, 2013.

\end{thebibliography}

\end{document}